\documentclass[11pt]{amsart}
\usepackage[utf8]{inputenc}
\usepackage{amsmath}
\usepackage[colorlinks]{hyperref}
\usepackage{amsthm}
\usepackage{amssymb}
\usepackage[left=1.5in,right=1.5in]{geometry}
\usepackage{todonotes}
\usepackage{tikz}
\usepackage{tikz-cd}
\usepackage{mathtools}
\usepackage{enumerate}
\usepackage{amsfonts}
\usepackage{mathrsfs}
\usepackage{thmtools}
\usepackage{cleveref}

\newtheorem{theorem}{Theorem}
\numberwithin{theorem}{section}
\newtheorem{lemma}[theorem]{Lemma}
\newtheorem{proposition}[theorem]{Proposition}
\newtheorem{conjecture}[theorem]{Conjecture}
\newtheorem{corollary}[theorem]{Corollary}

\theoremstyle{definition}
\newtheorem{definition}[theorem]{Definition}
\newtheorem{example}[theorem]{Example}

\theoremstyle{remark}
\newtheorem{remark}[theorem]{Remark}

\Crefname{conjecture}{Conjecture}{Conjectures}

\newcommand{\tR}{\widetilde{R}}
\newcommand{\cR}{\mathcal{R}}
\newcommand{\cC}{\mathcal{C}}

\newcommand{\RR}{\mathbb{R}}
\newcommand{\CC}{\mathbb{C}}

\DeclareMathOperator{\Ext}{Ext}

\usetikzlibrary{decorations.markings, arrows.meta}
\tikzcdset{
 arrow style=tikz,
 diagrams={>={Straight Barb[scale=0.8]}},
 marrow/.style={decoration={markings,mark=at position 0.5 with {\arrow{#1}}}, postaction=decorate},
 rmarrow/.style={decoration={markings,mark=at position 0.8 with {\arrow{#1}}}, postaction=decorate}
}

\newcommand{\newword}[1]{\textcolor{red!50!black}{\emph{#1}}}

\title[Combinatorial invariance for the coefficient of $q$]{Combinatorial invariance for the coefficient of $q$ in Kazhdan--Lusztig polynomials}
\author{Grant T. Barkley}
\author{Christian Gaetz}
\address[Barkley]{Department of Mathematics, University of Michigan, Ann Arbor, MI.}
\email{{\href{mailto:gbarkley@umich.edu}{gbarkley@umich.edu}}}
\address[Gaetz]{Department of Mathematics, University of California, Berkeley, CA.}
\email{{\href{mailto:gaetz@berkeley.edu}{{\tt gaetz@berkeley.edu}}}}
\author{Thomas Lam}
\address[Lam]{Department of Mathematics, University of Michigan, Ann Arbor, MI.}
\email{{\href{mailto:tfylam@umich.edu}{tfylam@umich.edu}}}
\date{\today}

\begin{document}
\begin{abstract}
We prove the combinatorial invariance of the coefficient of $q$ in Kazhdan--Lusztig polynomials for arbitrary Coxeter groups. As a result, we obtain the Combinatorial Invariance Conjecture, of Lusztig and of Dyer, also for Bruhat intervals of length at most $6$. We also prove the Gabber--Joseph conjecture for the second-highest $\Ext$ group of a pair of Verma modules, as well as the combinatorial invariance of the dimension of this group, and of the numbers of frozen and of mutable variables in the cluster structure on Richardson varieties.
\end{abstract}
\keywords{}
\subjclass{}
\maketitle
\section{Introduction}

\subsection{The Combinatorial Invariance Conjecture}
The \emph{Kazhdan--Lusztig polynomials} $P_{u,v}(q) \in \mathbb{N}[q]$, indexed by elements $u,v$ in a Coxeter group $W$, underlie deep connections between the singularities of Schubert varieties, representations of Lie algebras, bases for Hecke algebras, and Soergel bimodules \cite{Kazhdan-Lusztig-1, Elias-Williamson}. Yet the \emph{Combinatorial Invariance Conjecture}, due to Lusztig (c. 1983) and to Dyer \cite{Dyer1987}, asserts that they depend only on the Bruhat interval $[u,v]$ as an abstract poset (a property we will more generally refer to as \emph{combinatorial invariance}).

\begin{conjecture}[Combinatorial Invariance Conjecture]
\label{conj:cic}
Let $[u,v] \subset W$ and $[u',v'] \subset W'$ be Bruhat intervals such that $[u,v] \cong [u',v']$ as posets. Then $P_{u,v}(q)=P_{u',v'}(q)$.
\end{conjecture}

\noindent \Cref{conj:cic} can equivalently be stated also for the \emph{$R$-polynomials} $R_{u,v}(q)$, which in crystallographic type are the $\mathbb{F}_q$-point counts of the open Richardson varieties $\cR_{u,v}$, or for Brenti and Dyer's $\tR_{u,v}$ polynomials (see \Cref{sec:background}).

Despite significant attention, progress on \Cref{conj:cic} for general Coxeter groups has been limited. It follows from the definitions \cite{Kazhdan-Lusztig-1} that $P_{u,v}(q)$ has constant term $1$ and that $R_{u,v}(q)$ and $\tR_{u,v}(q)$ are monic of degree $\ell(u,v)$. And it follows from Dyer's interpretation \cite{Dyer1987} of $\tR_{u,v}(q)$ that the coefficient of $q^2$ in $\tR_{u,v}(q)$ is a combinatorial invariant; this implies \Cref{conj:cic} when $\ell(u,v) \leq 4$. To date, the Combinatorial Invariance Conjecture is not known to hold for any other coefficients or interval lengths, the most recent result for general Coxeter groups being the resolution  of the conjecture when $u=e$ and $u'=e'$ are both the identity elements of their respective groups \cite{Lower-intervals-general-type}. 

More progress has been made in the case $W$ is finite, and in particular when $W=S_n$ is the symmetric group, although the conjecture remains open even there. In this case it is known that there are finitely many isomorphism types of intervals of a fixed length; this allowed Incitti \cite{Incitti, Incitti2007} to verify \Cref{conj:cic} for intervals with $\ell(u,v) \leq 8$ in $S_n$. The theory of \emph{flipclasses} has also been applied to prove the combinatorial invariance of coefficients of small powers of $q$ in $\tR_{u,v}$, by reducing to a finite check \cite{esposito-marietti-flipclasses, esposito-marietti-stella}. Work for $W=S_n$ has recently been reinvigorated by the conjectural approach via \emph{hypercube decompositions} \cite{Blundell, davies}, which have been used to generalize the $u=e$ case to \emph{elementary intervals} \cite{elementary-paper}. See \cite{Brenti-open-problems} for a survey of progress on \Cref{conj:cic}.

For elements $u,v\in W$, the \emph{$d$-invariant} $d_{u,v} \in \mathbb{N}$ is defined by 
\[ R_{u,v} = q^{\ell(u,v)} - d_{u,v} q^{\ell(u,v)-1} + \cdots. \] 
This quantity is also of significant recent interest \cite{ellenberg2026bruhatintervalslargehypercubes}. Patimo \cite{Patimo-q-coefficient} introduced the quantity $g_{u,v}$, the minimum size of a \emph{diamond-generating} set of $[u,v]$ (see \Cref{sec:background}). Importantly, $g_{u,v}$ is manifestly an invariant of the poset $[u,v]$. The following is our main theorem.

\begin{theorem}
\label{thm:d-equals-g}
Let $u \leq v$ be elements of any Coxeter group $W$. Then $d_{u,v}=g_{u,v}$. 
\end{theorem}

\Cref{thm:d-equals-g} was proven by Patimo \cite{Patimo-q-coefficient} in the case where $W$ is finite and simply-laced. His argument relied on the \emph{generalized lifting property} \cite{Tsukerman-Williams} which is known \cite{caselli-sentinelli-gen-lifting} to hold only in this case. Our proof in \Cref{sec:proof} introduces new techniques using divergences of increasing paths to circumvent this.

\begin{corollary}
\label{cor:cic-for-coefficients}
Let $u \leq v \in W$ and $u' \leq v' \in W'$ be such that $[u,v] \cong [u',v']$. Then $[q]P_{u,v}=[q]P_{u',v'}$, $[q]R_{u,v}=[q]R_{u',v'}$, $[q^{\ell(u,v)-1}]R_{u,v}=[q^{\ell(u',v')-1}]R_{u',v'}$, and $[q^{\ell(u,v)}-2]\tR_{u,v}=[q^{\ell(u',v')}-2]\tR_{u',v'}$.
\end{corollary}

Although $P_{u,v}$ can have a quadratic term when $5 \leq \ell(u,v)\leq 6$, we are nevertheless able to prove \Cref{conj:cic} for such intervals.

\begin{corollary}
\label{cor:cic-length-5-6}
Let $u \leq v$ be elements of any Coxeter group $W$ and suppose $\ell(u,v) \leq 6$. Let $u' \leq v'$ be elements of a Coxeter group $W'$ such that $[u,v] \cong [u',v']$. Then $P_{u,v}=P_{u',v'}$, $R_{u,v}=R_{u',v'}$, and $\tR_{u,v}=\tR_{u',v'}$. 
\end{corollary}

\begin{table}
    \centering
    \begin{tabular}{c|c|c|c|c|c|c|c}
         &  $q^0$ & $q^1$ & $q^2$ & $\cdots$ & $q^{\ell(u,v)-2}$ & $q^{\ell(u,v)-1}$ & $q^{\ell(u,v)}$  \\
         \hline
       $P_{u,v}$  & \cite{Kazhdan-Lusztig-1} & \Cref{thm:d-equals-g} & ? & ? & $-$ & $-$ & $-$  \\
       \hline 
       $R_{u,v}$  & \cite{Kazhdan-Lusztig-1} & \Cref{thm:d-equals-g} & ? & ? & ? & \Cref{thm:d-equals-g} & \cite{Kazhdan-Lusztig-1}  \\
       \hline
       $\tR_{u,v}$ & \cite{Kazhdan-Lusztig-1} & \cite{Dyer1987} & \cite{Dyer1987} & ? & \Cref{thm:d-equals-g} & $-$ & \cite{Kazhdan-Lusztig-1}    \\
       \hline
    \end{tabular}
    \caption{Coefficients of the $P$-, $R$-, and $\tR$-polynomials which are known to be combinatorially invariant for all Coxeter groups. 
    } 
    \label{tab:knowncoefficients}
\end{table}

\begin{remark}
Recent results \cite{casals-gorsky-gorsky-le-shen-simental, galashin2024braidvarietyclusterstructuresI,galashin2025braidvarietyclusterstructures} have established that the coordinate ring of a Richardson variety is a \emph{cluster algebra} \cite{fomin-zelevinsky-cluster-algebras-1}. Although we do not need this fact, it is interesting to note that, when $W$ is a finite Weyl group, the specific diamond-generating set that we construct in the proof of \Cref{thm:d-equals-g} corresponds to the \emph{frozen variables} in this cluster structure. Thus we also obtain the combinatorial invariance of the numbers of frozen and of mutable variables in this cluster structure (see \Cref{sec:cluster}).
\end{remark}

\subsection{Extensions of Verma modules}
Let $\mathfrak{g}$ be a complex semisimple Lie algebra with Borel subalgebra $\mathfrak{b}$ and Weyl group $W$. For $w \in W$, let $M_w$ be the Verma module of highest weight $w\lambda-\rho$, where $\lambda$ is any anti-dominant regular integral weight and $\rho$ is the half-sum of the positive roots. Inspired by \cite{gabber-joseph}, the \emph{Gabber--Joseph conjecture}\footnote{Despite this standard name, the conjecture does not actually appear in the work of Gabber--Joseph which inspired it.} asserted that for $u \leq v$ the polynomial
\[
R'_{u,v}(q) \coloneqq \sum_{i=0}^{\ell(u,v)} (-1)^{\ell(u,v)-i} q^i \dim \Ext_{\mathcal{O}}^i(M_u,M_v),
\] 
is in fact equal to $R_{u,v}(q)$, where $\Ext_{\mathcal{O}}$ denotes extensions in the BGG category $\mathcal{O}$. It was shown by Verma in \cite{Verma} that $[q^0]R'_{u,v}=[q^0]R_{u,v}$ and it was shown by Carlin in \cite{Carlin} that $[q^{\ell(u,v)}]R'_{u,v} = [q^{\ell(u,v)}]R_{u,v}$. A counterexample to the full conjecture was eventually found by Boe \cite{Boe}. However, we prove the assertion for the coefficient of $q^{\ell(u,v)-1}$, and thereby also establish the combinatorial invariance of the corresponding $\Ext$ group.

\begin{theorem}
\label{thm:gj}
Let $u \leq v$ be elements of the Weyl group $W$ of a complex semisimple Lie algebra $\mathfrak g$. Then 
\[
\dim \Ext_{\mathcal{O}}^{\ell(u,v)-1}(M_u,M_v)=d_{u,v}. 
\]
Consequently, $\dim \Ext_{\mathcal{O}}^{\ell(u,v)-1}(M_u,M_v)$ is a combinatorial invariant of the poset $[u,v]$.
\end{theorem}

We prove \Cref{thm:gj} by establishing new vanishing results for the mixed Hodge structure of open Richardson varieties (see \Cref{prop:van}), which may be of independent interest.

\subsection{Outline}
In \Cref{sec:background} we recall background on Kazhdan--Lusztig polynomials, Bruhat graphs and reflection orders, and Richardson varieties. In \Cref{sec:proof} we introduce \emph{divergences} of increasing paths and prove \Cref{thm:d-equals-g}. In \Cref{sec:gj} we combine this with some geometric arguments to prove \Cref{thm:gj}. In \Cref{sec:cluster} we relate our results to the cluster structure on Richardson varieties, and derive some corollaries.

\section{Preliminaries}
\label{sec:background}

Let $W$ be a Coxeter group with simple generators $S$ and reflections $T$. We write $\ell$ for the length function on $W$ and $\leq$ for Bruhat order.

\subsection{Kazhdan--Lusztig polynomials} We now define the $R$-, $\tR$-, and $P$-polynomials.

\begin{theorem}[Kazhdan--Lusztig \cite{Kazhdan-Lusztig-1}]
    There is a unique family of polynomials $(R_{u,v})_{u,v\in W}$ from $\mathbb{Z}[q]$ so that for any $u,v\in W$, the following hold:
    \begin{itemize}
        \item If $u\not\leq v$, then $R_{u,v}=0$.
        \item If $u=v$, then $R_{u,v}=1$.
        \item If $u<v$, and $s\in S$ is such that $sv<v$, then
        \[ R_{u,v} = \begin{cases}
            R_{su,sv} &\text{if $su<u$} \\ 
            (q-1)R_{u,sv} + qR_{su,sv} &\text{if $su>u$}
        \end{cases}. \]
    \end{itemize}
\end{theorem}
As noted by Brenti and by Dyer, it follows from the form of this recurrence that there are unique polynomials $\tR_{u,v} \in \mathbb{N}[q]$ such that $R_{u,v}(q) = q^{\frac{\ell(u,v)}{2}} \tR_{u,v}(q^{1/2}-q^{-1/2})$. 
\begin{theorem}[Kazhdan--Lusztig \cite{Kazhdan-Lusztig-1}] \label{thm:KLexistence}
    There is a unique family of polynomials $(P_{u,v})_{u,v\in W}$ so that for any $u,v\in W$, the following hold:
    \begin{itemize}
        \item If $u\not\leq v$, then $P_{u,v}=0$.
        \item If $u=v$, then $P_{u,v}=1$.
        \item If $u<v$, then $\deg P_{u,v} < \frac{1}{2}\ell(u,v)$, and
        \[ q^{\ell(u,v)}P_{u,v}(q^{-1})-P_{u,v}(q) = \sum_{x\in (u,v]} R_{u,x}P_{x,v}. \]
    \end{itemize}
\end{theorem}

We write $[q^a]f$ for the coefficient of $q^a$ in a polynomial $f(q)$.

\subsection{Reflection orders and the Bruhat graph}
Let $\Phi$ be a (reduced) root system for $W$, spanning a vector space $\RR\Phi$, such that the simple roots $\Pi=\{\alpha_s \mid s \in S\}$ form a basis for $\RR\Phi$. Then $\Phi$ decomposes into positive roots $\Phi^+$ and their negations. 

\begin{definition}
    The \newword{Bruhat graph} $\Gamma$ is the directed, edge-labeled graph, with vertex set $W$ and an edge $x\to y$ labeled by $\beta\in \Phi^+$ whenever $t_\beta x = y$ and $\ell(x,y)>0$. We write $\Gamma_{u,v}$ for the induced subgraph on the elements of the Bruhat interval $[u,v]$, $\widehat{\Gamma}_{u,v}$ for its underlying undirected graph, and $E_{u,v}$ for the set of edges of $\widehat{\Gamma}_{u,v}$.
\end{definition}

Dyer used certain orderings of $\Phi^+$ to study the $\tR$-polynomials.

\begin{definition}
    A \newword{reflection order} is a total order $\prec$ on $\Phi^+$ so that, for any triple of roots $\alpha,\beta,\gamma\in \Phi^+$ with $\gamma \in \RR_{>0}\alpha+\RR_{>0}\beta$, we have either 
    \[ \alpha \prec \gamma \prec \beta \qquad\text{or}\qquad \beta \prec \gamma \prec \alpha. \]
\end{definition}

We say a (directed) path $\gamma = (x_0\xrightarrow{\beta_1} x_1 \xrightarrow{\beta_2}\cdots \xrightarrow{\beta_k}x_k)$ in the Bruhat graph $\Gamma$ is an \newword{increasing path} from $x_0$ to $x_k$ (with respect to a reflection order $\prec$) if the edge labels satisfy $\beta_1\prec \beta_2 \prec \cdots \prec \beta_k$. The \newword{length} of $\gamma$ is $\ell(\gamma)\coloneqq k$.

\begin{proposition}[\cite{Dyer1987}]
\label{prop:Rpolyincreasingpaths}
    Fix a reflection order $\prec$. Then for any $u,v\in W$, 
    \[ \tR_{u,v}(q)  = \sum_{\substack{\gamma}} q^{\ell(\gamma)}, \]
    where the sum is over increasing paths $\gamma$ from $u$ to $v$.
\end{proposition}

An \newword{initial section} of a reflection order is a set $X \subseteq \Phi^+$ so that if $\beta\in X$ and $\alpha\in \Phi^+$ with $\alpha\prec \beta$, then $\alpha\in X$. An \newword{inversion set} is a set of roots of the form $N(w)\coloneqq \{\beta\in \Phi^+ \mid w^{-1}\beta \not\in \Phi^+\}$ for some $w\in W$. If $s \in S$ and $\ell(w)<\ell(ws)$, then $N(ws) =  N(w)\sqcup \{ w\alpha_{s} \}$. In this way, a reduced expression $w=s_{1}\cdots s_{\ell}$ is equivalent to a chain of inversion sets $\varnothing \subset N(s_{1}) \subset N(s_{1}s_{2}) \subset \cdots \subset N(w)$.

\begin{proposition}[{\cite{Dyer1993}}]
\label{prop:existreflectionorders}
    Given any chain of inversion sets $X_1 \subset X_2 \subset \cdots \subset X_k$, there exists a reflection order $\prec$ so that, for all $i$, $X_i$ is an initial section of $\prec$.
\end{proposition}

There is another source of reflection orders which we will also use (see, e.g., \cite[Ch.~5]{Bjorner-Brenti}).
\begin{definition}\label{def:uppersconjugate}
    Given a reflection order $\prec$ and a simple reflection $s\in S$, the \newword{upper $s$-conjugate} is the reflection order $\prec^s$ such that, for distinct positive roots $\alpha,\beta$, we have $\alpha\prec^s \beta$ if and only if one of the following hold:
    \begin{itemize}
        \item $\alpha,\beta\prec \alpha_s$ and $\alpha\prec \beta$; or
        \item $\alpha \prec \alpha_s \prec \beta$; or
        \item $\alpha_s \prec \alpha,\beta$ and $s\alpha \prec s\beta$; or
        \item $\beta = \alpha_s$.
    \end{itemize}
\end{definition}

We will make use of the upper $s$-conjugate via the following (well-known) lemma, whose proof we include for completeness.

\begin{lemma}\label{lem:sgamma}
    Fix a reflection order $\prec$, and let $s\in S$ be such that $\alpha_s$ is the $\prec$-minimal positive root. Let $\gamma = (u=x_0\xrightarrow{\beta_1} x_1 \xrightarrow{\beta_2} \cdots \xrightarrow{\beta_r} x_{r}=v)$ be an increasing path from $u$ to $v$ such that $su$ is not a vertex of $\gamma$. Then 
    \[ s\gamma \coloneqq (sx_0 \to sx_1 \to \cdots \to sx_r) \]
    is a path from $su$ to $sv$ which is increasing with respect to $\prec^s$.
\end{lemma}
\begin{proof}
    Because $\alpha_s$ is the $\prec$-minimal root and $\gamma$ is an increasing path, either $x_1 = su$ or else $sx_i \neq x_{i+1}$ for all $i$. By hypothesis, we are in the latter case. Then the lifting property for Bruhat order (see, e.g., \cite[Proposition 2.2.7]{Bjorner-Brenti}) implies that $sx_i < sx_{i+1}$ for all $i$, so $s\gamma$ is indeed a path from $su$ to $sv$. Moreover, the edge label of $sx_i \to sx_{i+1}$ is $s\beta_{i+1}$. Since $\alpha_s\prec \beta_i \prec \beta_{i+1}$, we have $s\beta_i \prec^s s\beta_{i+1}$. Hence $s\gamma$ is increasing with respect to $\prec^s$.
\end{proof}

\subsection{The $d$-invariant}

For elements $u,v\in W$, define $d_{u,v} = -[q^{\ell(u,v)-1}]R_{u,v}$, so that $R_{u,v} = q^{\ell(u,v)} - d_{u,v} q^{\ell(u,v)-1} + \cdots.$ The recurrence for $R$-polynomials implies a recurrence for the $d$-invariant. 

\begin{lemma}\label{lem:drecurrence}
    If $s\in S$ and $u,v\in W $ are such that $sv<v$, then 
    \[ d_{u,v} = \begin{cases}
        d_{su,sv} &\text{if $su<u$,}\\
        d_{u,sv}+1 &\text{if $su>u$ and $su\not \leq sv$,} \\
        d_{u,sv} &\text{if $su>u$ and $su\leq sv$.}
    \end{cases} \]
\end{lemma}

The definitions of the $P$-, $R$-, and $\tR$-polynomials imply that $d_{u,v}$ is related to various of their coefficients.
\begin{lemma}[See e.g. \cite{dyer-q-coefficient}]\label{lem:dincarnations}
    For any elements $u\leq v$ in $W$ we have:
    \begin{itemize}
        \item[(a)] $[q]P_{u,v} = |\{c\in [u,v] \mid \ell(c,v) =1 \}| - d_{u,v}$.
        \item[(b)] $[q^{\ell(u,v)-1}]R_{u,v} = -d_{u,v}$.
        \item[(c)] $[q]R_{u,v} = (-1)^{\ell(u,v)-1} d_{u,v}$.
        \item[(d)] $[q^{\ell(u,v)-2}]\tR_{u,v} = \ell(u,v) - d_{u,v}$.
    \end{itemize}
\end{lemma}

Our main result will be that $d_{u,v}$ is equal to a quantity $g_{u,v}$, introduced by Patimo in \cite{Patimo-q-coefficient}, which is manifestly an invariant of the poset $[u,v]$. 

\begin{definition}[Def.~4.5 of \cite{Patimo-q-coefficient}]
We say that $F \subseteq E_{u,v}$ is \newword{diamond-closed} if, whenever it contains two consecutive edges in a $4$-cycle of $\widehat{\Gamma}_{u,v}$, it also contains the other two. We denote by $F^{\diamond}$ the \newword{diamond-closure} of $F$: the smallest diamond-closed subset of $E_{u,v}$ such that $F \subseteq F^{\diamond}$. We call $F$ \newword{diamond-generating} if $F^{\diamond}=E_{u,v}$. We write $g_{u,v}$ for the minimum size of a diamond-generating subset of $E_{u,v}$.
\end{definition}

The quantity $g_{u,v}$ is an invariant of the poset $[u,v]$ since the poset structure determines $\Gamma_{u,v}$ as an unlabeled directed graph \cite{Dyer-bruhat-graph}.

\begin{figure}
    \centering
    	\begin{tikzpicture}[scale=1]
  
			\node (a) at (0,0) {$u$}; 
			\node (b1) at (-1.5,1) {$\bullet$}; 
			\node (b2) at (-0.5,1) {$\bullet$}; 
			\node (b3) at (0.5,1) {$\bullet$}; 
			\node (b4) at (1.5,1) {$\bullet$}; 
			\node (c1) at (-1.5,2) {$\bullet$}; 
			\node (c2) at (-0.5,2) {$\bullet$}; 
			\node (c3) at (0.5,2) {$\bullet$}; 
			\node (c4) at (1.5,2) {$\bullet$}; 
			\node (d) at (0,3) {$v$}; 
			\draw[->] (a)--(b4);
			\draw[->] (c1)--(d);
			\draw[->] (c2)--(d);
			\draw[->] (c3)--(d);
			\draw[->] (c4)--(d);
			\draw[->] (b1)--(c1);
			\draw[->] (b1)--(c2);
			\draw[->] (b2)--(c1);
			\draw[->] (b2)--(c3);
			\draw[->] (b3)--(c2);
			\draw[->] (b3)--(c4);
			\draw[->] (b4)--(c3);
			\draw[->] (b4)--(c4);
            \draw[->,red] (a)--(b1);
			\draw[->,red] (a)--(b2);
			\draw[->,red] (a)--(b3);
    \end{tikzpicture}
    \caption{A Bruhat graph $\Gamma_{u,v}$ and a minimum-size diamond-generating set (shown in red).}
    \label{fig:placeholder}
\end{figure}
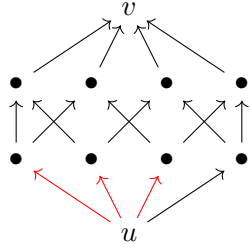

\begin{remark}
    We have defined $g_{u,v}$ as the minimum size of a diamond-generating set. One could instead consider variants of this definition. Here are two similar invariants one could define.
    
    For each pair $x,y\in [u,v]$ such that $x\leq y$ and $\ell(x,y)=2$, the graph $\widehat\Gamma_{x,y}$ is a $4$-cycle in $\widehat\Gamma_{u,v}$. Define a \emph{weakly diamond-closed} set $F\subseteq E_{u,v}$ to be one such that, whenever $F$ contains two consecutive edges in $\widehat\Gamma_{x,y}$ for some $x,y\in [u,v]$ with $\ell(x,y)=2$, then $F$ also contains the other two edges of $\widehat\Gamma_{x,y}$.
    \begin{itemize}
        \item Let $g_{u,v}'$ be the minimum size of a diamond-generating set $F$ such that every edge in $F$ has length 1.
        \item Let $g_{u,v}''$ be the minimum size of a weakly diamond-generating set $F$.
    \end{itemize}
    \emph{A priori}, we have $g_{u,v} \leq g_{u,v}' \leq g_{u,v}''$. It follows from our proof of \Cref{thm:d-equals-g} that these alternative definitions in fact define the same invariant: that is, $g_{u,v}=g_{u,v}'=g_{u,v}''=d_{u,v}$.
\end{remark}

The following lemma is immediate from the shellability \cite{bjorner-wachs} of order complexes of Bruhat intervals.

\begin{lemma}\label{lem:chainsdiamondgenerate}
    If $F$ contains the set of edges of a maximal chain of $[u,v]$, then $F$ diamond-generates $[u,v]$.
\end{lemma}

Patimo \cite{Patimo-q-coefficient} showed the following, which also follows from Dyer's work \cite{dyer-q-coefficient}. \Cref{thm:d-equals-g} establishes equality.

\begin{lemma}[Prop.~4.7 of \cite{Patimo-q-coefficient}]\label{lem:dginequality}
    For any $u,v\in W$, we have $d_{u,v}\leq g_{u,v}$.
\end{lemma}

\subsection{Richardson varieties}\label{sec:defRichardson}

Let $G$ be a (minimal) Kac--Moody group with Borel subgroup $B$ and Weyl group $W$. For $u \leq v$ in $W$, the \newword{(open) Richardson variety} is 
\[
\cR_{u,v} \coloneqq (BvB/B) \cap (B^{-}uB/B),
\]
where $B^-$ denotes the opposite Borel. This is an irreducible $\ell(u,v)$-dimensional subvariety of the ind-variety $G/B$. It is well known that $R_{u,v}(q)$ gives the $\mathbb{F}_q$-point count of $\cR_{u,v}$.

\section{Divergences of increasing paths}\label{sec:proof}

\begin{proof}[Proof of \Cref{thm:d-equals-g}]
By \Cref{lem:dginequality}, there is an inequality $d_{u,v}\leq g_{u,v}$. Hence to show that $d_{u,v} = g_{u,v}$, it is enough to exhibit a diamond-generating set for $[u,v]$ of size $d_{u,v}$. This is accomplished by \Cref{thm:gchain}, the statement and proof of which will occupy the remainder of the section.
\end{proof}

Let $\prec$ be a reflection order. By \Cref{prop:Rpolyincreasingpaths} and the monicity of $\tR_{u,v}$, when $u\leq v$ there is a unique increasing path $\gamma$ of length $\ell(u,v)$ from $u$ to $v$. We call this the \newword{longest (increasing) path} from $u$ to $v$, denoted $\gamma_0$. An increasing path of length $\ell(u,v)-2$ is called a \newword{second-length (increasing) path} from $u$ to $v$.
We say that a second-length path $\gamma$ \newword{diverges (from $\gamma_0$) at $x$} if each vertex of $\gamma_0$ in $[u,x]$ is a vertex of $\gamma$ and, if $x'$ denotes the vertex of $\gamma_0$ covering $x$, then $x'$ is not on $\gamma$.

\begin{lemma}\label{lem:countdiverging}
    Let $\prec$ be a reflection order and let $u<v$ be elements of $W$. Let $u\to x_1$ be the first edge on the longest path $\gamma_0$ from $u$ to $v$. Then the number of second-length paths diverging from $\gamma_0$ at $u$ is $d_{x_1,v}-d_{u,v}+1$. In particular, $d_{u,v}-d_{x_1,v}\leq 1$.
\end{lemma}
\begin{proof}
    Let $a_{x,y}$ denote the number of second-length paths in $[x,y]$ (which, by \Cref{prop:Rpolyincreasingpaths}, is equal to $[q^{\ell(x,y)-2}]\tR_{x,y}$). Let $u\to x_1$ be the first edge of $\gamma_0$. Evidently, any second-length path from $u$ to $v$ which does not diverge at $u$ must be the concatenation of the edge $u\to x_1$ with a second-length path from $x_1$ to $v$. Conversely,
    it follows from the fact that $\gamma_0$ is the \emph{lex-minimal} path from $u$ to $v$ (see, e.g., \cite[Proposition 4.3]{Dyer1993}) that any second-length path from $x_1$ to $v$ concatenates with $u\to x_1$ to give a second-length path from $u$ to $v$. Hence the number of second-length paths for $[u,v]$ which diverge at $u$ is $a_{u,v}-a_{x_1,v}$.

    Now, by \Cref{lem:dincarnations}(d), we have the identity
    \[ a_{u,v} - a_{x_1,v} = d_{x_1,v}-d_{u,v}+1, \]
    and the claim follows.
\end{proof}

\begin{remark}
    It is possible for $d_{u,v}-d_{x_1,v}$ to be negative. For example, taking $W=S_4$ and the elements $u=e$ and $v=s_1s_2s_3s_2s_1$, we have that
    \begin{align*}
        d_{e,v} &= 3 \\
        d_{s_2,v} &= 4.
    \end{align*}
    Any reflection order so that $\alpha_2$ is minimal will put $e\xrightarrow{\alpha_2} s_2$ as the first edge of $\gamma_0$. Hence $d_{u,v}-d_{x_1,v} = -1$ for such reflection orders.
    
    As we will see, it is always possible to choose the reflection order so as to avoid this situation.
\end{remark}
\begin{definition}
    Let $v\in W$. A reflection order $\prec$ is said to be \newword{Deodhar with respect to $v$} if $N(v)$ is an initial section of $\prec$.
\end{definition}

For any $v\in W$, there exist reflection orders which are Deodhar with respect to $v$, by \Cref{prop:existreflectionorders}.

\begin{lemma}\label{lem:ddecreasing}
    Let $u\leq v$ be elements of $W$ and let $\prec$ be a reflection order which is Deodhar with respect to $v$. Let $\gamma_0=(u=x_0\to x_1 \to\cdots \to x_{\ell(u,v)}=v)$ be the longest path from $u$ to $v$. Then $d_{x_i,v} \geq d_{x_{i+1},v}$ for all $i$. 
\end{lemma}
\begin{proof}

If $v$ is the identity element, then this is trivial. Otherwise, if $\alpha_s$ is the $\prec$-minimal positive root, then $sv<v$. By (two applications of) \Cref{lem:drecurrence}, we have for any $x \in [u,v]$ that
\[ d_{x,v} = \begin{cases}
        d_{sx,sv} &\text{if $sx<x$}\\
        d_{sx,v}+1 &\text{if $sx>x$ and $sx\not \leq sv$} \\
        d_{sx,v} &\text{if $sx>x$ and $sx\leq sv$}.
    \end{cases} \]
    
Observe that, since $\alpha_s$ is the $\prec$-minimal positive root, the longest path $\gamma_0$ either begins with an edge of the form $u \xrightarrow{\alpha_s} su$ or else contains no edges labeled with $\alpha_s$. Moreover, since $\gamma_0$ is the lex-minimal path from $u$ to $v$, the first case occurs if and only if $su>u$. Hence in the first case, $d_{u,v}$ is equal to either $d_{su,v}+1$ or $d_{su,v}$, as claimed. So assume we are in the second case.

Then \Cref{lem:sgamma} implies that $sx_i < x_i$ for all $i$. It follows that $d_{x_i,v} = d_{sx_i,sv}$ for all $i$. Moreover, \Cref{lem:sgamma} asserts that the path $(sx_0 \to sx_1 \to \cdots \to sx_k)$ is the longest path from $su$ to $sv$ with respect to the reflection order $\prec^s$. Hence the claim follows by induction, since $\prec^s$ is Deodhar with respect to $sv$.
\end{proof}

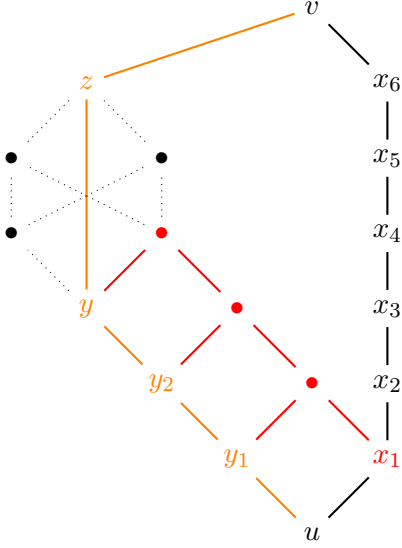
\begin{figure}
    \centering
    \begin{tikzpicture}
        \node (x0) at (0,0) {$u$};
        \node[red] (x1) at (1,1) {$x_1$};
        \foreach \i in {2,...,6} {
        \node (x\i) at (1,\i) {$x_{\i}$};
        }
        \node (x7) at (0,7) {$v$};
        \draw[thick] (x0) -- (x1) -- (x2) -- (x3) -- (x4) -- (x5) -- (x6) -- (x7);

        \node[orange] (y1) at (-1,1) {$y_1$};
        \node[orange] (y2) at (-2,2) {$y_2$};
        \node[orange] (x) at (-3,3) {$y$};
        \node[orange] (y) at (-3,6) {$z$};
        \draw[orange, thick] (x0) -- (y1) -- (y2) -- (x) -- (y) -- (x7);

        \node[red] (c1) at (0,2) {$\bullet$};
        \node[red] (c2) at (-1,3) {$\bullet$};
        \node[red] (c3) at (-2,4) {$\bullet$};
        \draw[red, thick] (x1) -- (c1) -- (c2) -- (c3);
        \draw[red, thick] (y1) edge (c1) (y2) edge (c2) (x) edge (c3);

        \node[black] (L1) at (-4,4) {$\bullet$};
        \node[black] (L2) at (-4,5) {$\bullet$};
        \node[black] (R2) at (-2,5) {$\bullet$};
        \draw[black, dotted] (x) -- (L1) -- (L2) -- (y);
        \draw[black, dotted] (c3) -- (R2) -- (y) (L1) edge (R2) (c3) edge (L2);
    \end{tikzpicture}
    \caption{An example of a supporting chain $\cC$. The longest path $\gamma_0$ is shown in black, the second-length path $\gamma$ is shown in orange, and $\cC$ is shown in red. The cover relations in $[y,z]$ are shown with dotted lines.}
    \label{fig:supportingchain}
\end{figure}

\begin{theorem}\label{thm:gchain}
Let $u\leq v$ be elements of $W$, and let $\prec$ be a reflection order which is Deodhar with respect to $v$. Let $\gamma_0=(u=x_0\to x_1 \to\cdots \to x_{\ell(u,v)}=v)$ be the longest path from $u$ to $v$. Define the set of edges
\[ F_{u,v} \coloneqq \{ (x_i,x_{i+1}) \mid d_{x_{i},v} = d_{x_{i+1},v} + 1  \}. \]
Then $F_{u,v}$ is a diamond-generating set for $[u,v]$ of size $d_{u,v}$.
\end{theorem}
\begin{proof}
The fact that $|F_{u,v}|=d_{u,v}$ follows from \Cref{lem:countdiverging,lem:ddecreasing}.

We will show that $F_{u,v}$ diamond-generates by induction on $\ell(u,v)$. If $\ell(u,v) = 0$ then we are done, so assume $\ell(u,v)\geq 1$. First note that $F_{u,v} \cap E_{x_1,v} = F_{x_1,v}$. Hence, by induction, $F_{x_1,v}$ is a diamond-generating set for $[x_1,v]$. By \Cref{lem:chainsdiamondgenerate}, $E_{x_1,v}\cup \{(u,x_1)\}$ diamond-generates $[u,v]$. Thus if $F_{u,v}$ contains $(u,x_1)$ (equivalently, if $d_{u,v} = d_{x_1,v}+1$), then $F_{u,v}$ is diamond-generating.

So assume that $d_{u,v} = d_{x_1,v}$ (which is the only other case, by \Cref{lem:ddecreasing}). In this case, we will show that the diamond-closure of $E_{x_1,v}$ in $[u,v]$ contains the edge $(u,x_1)$ and therefore (again by \Cref{lem:chainsdiamondgenerate}) that $E_{x_1,v}$ (and hence $F_{u,v}$) is diamond-generating for $[u,v]$.

Since $d_{u,v}=d_{x_1,v}$, there is a unique second-length path diverging at $u$ by \Cref{lem:countdiverging}. Denote this second-length path by $\gamma$. There is a unique edge of length 3 in $\gamma$, say $y \to z$. Define a \newword{supporting chain} to
be a saturated chain $\cC \subseteq [u,v]$ (meaning that cover relations in $\cC$ are also cover relations in $[u,v]$) such that: 
\begin{itemize}
    \item $\cC$ is disjoint from $\gamma$; and
    \item The minimal element of $\cC$ is $x_1$; and
    \item The maximal element of $\cC$ is an atom of $[y,z]$; and
    \item For each $c\in \cC$, there is a vertex of $\gamma$ which is covered by $c$.
\end{itemize}
See \Cref{fig:supportingchain} for an example of a supporting chain.

\noindent\textbf{Claim:} A supporting chain exists.
\begin{proof}[Proof of claim]
\renewcommand{\qedsymbol}{$\blacksquare$}
    Let the vertices of $\gamma$ in the interval $[u,y]$ be $u=y_0,y_1,\ldots,y_r=y$.

    Let $s\in S$ be such that $\alpha_s$ is the $\prec$-minimal positive root. Then $sv<v$. The fact that $\gamma$ diverges at $u$ implies that none of the edge labels $y_i\xrightarrow{\alpha} y_{i+1}$ are equal to $\alpha_s$. 
    
    First assume $su > u$. Using \Cref{lem:sgamma}, we see that $s\gamma$ is an increasing path from $su$ to $sv$ with respect to $\prec^s$. Since $\ell(s\gamma) = \ell(su,sv)$, all edges of $s\gamma$ must be length $1$. Since $\gamma$ has one edge of length $3$, it must be that $sw>w$ for all vertices $w$ of $\gamma$ in $[u,y]$, and $sw<w$ for all vertices $w$ of $\gamma$ in $[z,v]$. In particular, $sy_i>y_i$ for all $i$ and $sz<z$. By the lifting property for Bruhat order \cite[Proposition 2.2.7]{Bjorner-Brenti}, this implies $sy_r < z$. Hence the chain $\cC\coloneqq \{sy_0, sy_1,\ldots,sy_r\}$ is a supporting chain.

    We now assume instead that $su < u$. Using \Cref{lem:sgamma}, we see that both $s\gamma_0$ and $s\gamma$ are increasing paths from $su$ to $sv$ with respect to $\prec^s$. Moreover, $s\gamma_0$ is the longest path with respect to $\prec^s$, and $s\gamma$ is a second-length path diverging at $su$. As in the proof of \Cref{lem:ddecreasing}, we have $d_{x_i,x_{i+1}}=d_{sx_i,sx_{i+1}}$ for all $i$. In particular, $d_{su,sv}=d_{sx_1,sv}$. Hence by induction on $\ell(v)$, there is a supporting chain $\cC'$ for the path $s\gamma$ in $[su,sv]$.

\begin{figure}
    \centering
    \begin{tikzpicture}
        \node (x0) at (0,0) {$u$};
        \node[red] (x1) at (1,1) {$x_1$};
        \foreach \i in {2,...,5} {
        \node (x\i) at (1,\i) {$x_{\i}$};
        }
        \node (x6) at (0,7) {$v$};
        \draw[thick] (x0) -- (x1) -- (x2) -- (x3) -- (x4) -- (x5) -- (x6);

        \node[orange] (y1) at (-2,1) {$y_1$};
        \node[orange] (x) at (-2,2) {$y$};
        \node[orange] (y) at (-2,5) {$z$};
        \draw[orange, thick] (x0) -- (y1) -- (x) -- (y) -- (x7);

        \node (su) at (-1,-1) {$su$};
        \node (sy1) at (-3,0) {$sy_1$};
        \node[red] (sx) at (-3,3) {$sy$};
        \node (sy) at (-3,4) {$sz$};
        \node[purple] (sx1) at (-2,0) {$sx_1$};
        
        \draw[blue] (x0) edge (su) (y1) edge (sy1) (x) edge (sx) (y) edge (sy) (x1) edge (sx1);

        \draw[orange, dashed] (su) -- (sy1) -- (sx) -- (sy);

        \node[purple] (cp1) at (-4,1) {$\circ$};
        \draw[purple, dashed] (sx1) -- (cp1) -- (sy1);
        \node[red] (L2p) at (-4,2) {$\bullet$};
        \draw[blue] (cp1) -- (L2p);
        \draw[dotted] (su) edge (sx1) (cp1) edge (x) (y1) edge (L2p);

        \draw[red] (L2p) edge (sx);
        \draw[red] (x1) -- (L2p);


    \end{tikzpicture}
    \caption{
    The inductive construction of a supporting chain. The action of $s$ is indicated with blue edges. The second-length path for $[su,sv]$ is shown with dashed orange lines and the supporting chain $\cC'=\{\textcolor{purple}{sx_1,\circ}\}$ is shown in purple. The supporting chain for $[u,v]$ is $\cC=\{\textcolor{red}{x_1,\bullet,sy}\}$, shown in red. 
    }
    \label{fig:doubledescent}
\end{figure}
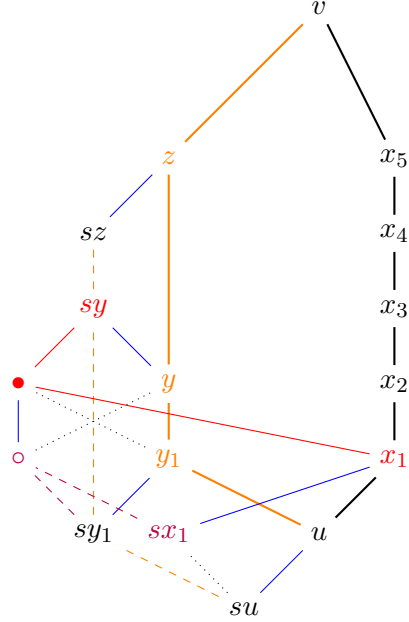

    If $sy<y$, then $sw<w$ for every vertex $w$ on $\gamma$. In this case, $\cC \coloneqq s\cC'$ is a supporting chain for $\gamma$. Otherwise, we are in the case where $sy>y$. This case is illustrated in \Cref{fig:doubledescent}. There is some index $j$ so that $sy_0<y_0, sy_1<y_1,\ldots, sy_j<y_j$ and $sy_{j+1}>y_{j+1},\ldots,sy_r>y_r$. In particular, $sy_j \to sy_{j+1}$ is the unique length 3 edge of $s\gamma$. We set $\cC \coloneqq s\cC' \sqcup \{sy_{j+1},\ldots, sy_r\}$. 

    Write the elements of $\cC'$ in order as $c_0,\ldots,c_k$.
    Assume to the contrary that there exist $c_i,c_{i+1}$ such that $sc_i>c_i$ and $sc_{i+1}<c_{i+1}$. This would imply, by the lifting property of Bruhat order, that $sc_i=c_{i+1}$. Moreover, since $c_i$ covers $sy_i$, the lifting property further implies that $s\cdot sy_i < sy_i$, which is a contradiction. Hence there can be no $i$ such that $sc_i < c_i$. From this we can see that $\cC$ is a saturated chain, since $\cC'\sqcup\{y_{j+1},\ldots,y_r\}$ is a saturated chain, each of whose elements is taken upward by $s$. Moreover $sy_r \leq z$, by using the lifting property with the fact that $sz<z$. Hence $\cC$ is a supporting chain.
\end{proof}

Now that we have a supporting chain $\cC$, it is easy to see (for instance, by examining \Cref{fig:supportingchain}) that because $\cC$ is contained in $[x_1,v]$, the diamond-closure of $E_{x_1,v}$ in $[u,v]$ includes every (cover-relation) edge between two elements of $\cC$, between an element of $\cC$ and a vertex of $\gamma$, and between two vertices of $\gamma$ that are below elements of $\cC$. In particular, the edge $(u,x_1)$ is in the diamond-closure, so $F_{u,v}$ is a diamond-generating set for $[u,v]$.
\end{proof}

\begin{proof}[Proof of \Cref{cor:cic-for-coefficients}]
    This follows from \Cref{thm:d-equals-g} and \Cref{lem:dincarnations}.
\end{proof}

\begin{proof}[Proof of \Cref{cor:cic-length-5-6}]
    \Cref{cor:cic-for-coefficients} gives combinatorial invariance for $[q^{\ell(u,v)-2}]\tR_{u,v}$. Since combinatorial invariance was also known for $[q^1]\tR_{u,v}, [q^2]\tR_{u,v},$ and $[q^{\ell(u,v)}]\tR_{u,v}$, it follows that the $\tR$-polynomial for an interval of length 5 or 6 is a combinatorial invariant (since all powers of $q$ appearing in $\tR_{u,v}$ have the same parity as $\ell(u,v)$). Since $R_{u,v}(q) = q^{\ell(u,v)/2}\tR_{u,v}(q^{1/2}-q^{-1/2})$ and (by \Cref{thm:KLexistence}) $P_{u,v}$ can be computed from the values of $R_{x,y}$ for $u<x\leq y \leq  v$, we deduce that the $R$- and $P$-polynomials of an interval of length 5 or 6 are also combinatorially invariant.
\end{proof}

\section{Application to Gabber--Joseph coefficients}
\label{sec:gj}

In this section, we return to the case where $W$ is the Weyl group of a symmetrizable Kac--Moody group $G$ with Richardson varieties $(\cR_{u,v})_{u,v\in W}$, as in \Cref{sec:defRichardson}.

The cohomology groups of a complex algebraic variety are equipped with a mixed Hodge structure, inducing a Deligne splitting that we denote $H^k = \bigoplus_{p,q} H^{k,(p,q)}$.  We briefly discuss the mixed Hodge structure of Richardson varieties.

\begin{proposition}\label{prop:van}
    Let $W$ be the Weyl group of a complex semisimple Lie group and $u,v \in W$. Then the mixed Hodge spaces $H^{1,(0,0)}(\cR_{u,v};\mathbb{C})$ and $H^{2,(1,1)}(\cR_{u,v};\mathbb{C})$ vanish.
\end{proposition}
\begin{proof}
By \cite[Lemma.~4.3.1]{riche-soergel-williamson}, there exists a decomposition $\cR_{u,v} = U \sqcup Z$ where $Z$ is a smooth hypersurface and $U$ is the (open) complement, and such that $U \cong \CC^\times \times \cR'$ and $Z \cong \CC \times \cR''$ for Richardson varieties $\cR', \cR''$ of lower dimension. It follows from the induction, the Gysin exact sequence (see below), and the decomposition $\cR_{u,v} = U \sqcup Z$ that Richardson varieties are mixed Tate, so we have $H^k(\cR_{u,v},\mathbb{C}) = \bigoplus_{0 \leq p \leq k} H^{k,(p,p)}(\cR_{u,v},\mathbb{C})$.  

The vanishing of $H^{1,(0,0)}(\cR_{u,v};\mathbb{C})$ holds since $\cR_{u,v}$ is smooth, which implies that 
\begin{equation}
    \label{eq:smooth}
H^k(\cR_{u,v},\mathbb{C}) = \bigoplus_{\lceil k/2 \rceil \leq p \leq k} H^{k,(p,p)}(\cR_{u,v},\mathbb{C}).
\end{equation}
We have
$$
H^2(\cR_{u,v}) = H^{2,(2,2)}(\cR_{u,v}) \oplus H^{2,(1,1)}(\cR_{u,v}).
$$

We proceed by induction on $\ell(u,v)$.  The cases $\ell = 0,1$ are trivial. We have the Gysin exact sequence
$$
H^1(U) \to H^0(Z) \to H^2(\cR_{u,v}) \to H^2(U) \to H^1(Z) \to H^3(\cR_{u,v}) \to
$$
The map $H^2(\cR_{u,v}) \to H^2(U)$ preserves the mixed Hodge decomposition, and 
\[
H^{2,(1,1)}(U) = H^{2,(1,1)}(\cR') \oplus H^{1,(1,1)}(\cR') = H^{2,(1,1)}(\cR') = 0,
\]
by induction.  Thus $H^{2,(1,1)}(\cR_{u,v})$ lies in the image of the map $H^0(Z) \to H^2(\cR_{u,v})$.  We show that the map $H^0(Z) \to H^2(\cR_{u,v})$ vanishes.  Since $Z$ is connected, $\dim H^0(Z) = 1$, and it suffices to show that $H^1(U) \to H^0(Z)$ is non-zero.  Let $f$ be the regular function on $\cR_{u,v}$ such that $Z = \{f = 0\}$ is cut out by $f$ to order one.  Then $df/f$ is a regular 1-form on $U$.  The image of the class $[df/f] \in H^1(U)$ (viewed as a cohomology class in algebraic de Rham cohomology) in $H^0(Z)$ is the class of the residue ${\rm Res}_{f = 0}(df/f)$, which is equal to 1.  This class spans $H^0(Z)$.  We conclude that $H^{2,(1,1)}(\cR_{u,v}) = 0$, completing the inductive step.
\end{proof}

\begin{remark}
    \Cref{prop:van} can also be deduced using properties of the mixed Hodge structures of cluster varieties \cite{lam-speyer} (Richardson varieties are known to be cluster varieties by \cite{galashin2025braidvarietyclusterstructures,casals-gorsky-gorsky-le-shen-simental}).
\end{remark}

We obtain the following corollary by interpreting $R_{u,v} = \#\cR_{u,v}(\mathbb{F}_q)$ as the point count of a Richardson variety over $\mathbb{F}_q$ \cite{Kazhdan-Lusztig-1}. 
\begin{corollary}\label{cor:duv}
Let $u \leq v$ be elements of a finite Weyl group $W$. Then
    \[ d_{u,v} = \dim H^{2\ell(u,v)-1}_c(\cR_{u,v};\mathbb{C}) = \dim H^{1}(\cR_{u,v};\mathbb{C}), \]
    where $H_c^\bullet$ denotes compactly supported singular cohomology, and $H^\bullet$ denotes singular cohomology. 
\end{corollary}
\begin{proof}
Let $(\cR_{u,v})_{\mathbb{F}_q}$ denote the Richardson variety defined over $\overline{\mathbb{F}}_q$.  The Frobenius eigenvalues of the \'etale cohomology groups $H^{k}_c((\cR_{u,v})_{\mathbb{F}_q};\overline{\mathbb{Q}}_\ell)$ are known to be of the form $q^a$, where $a$ is an integer.  (This can be shown using induction and the decompositions $\cR_{u,v} = U \sqcup Z$.)  

By the Grothendieck--Lefschetz trace formula, the coefficient of $q^{\ell(u,v)-1}$ in $R_{u,v}(q)$ is equal to an alternating count of the appearances of the Frobenius eigenvalue $q^{\ell(u,v)-1}$ in $H^\bullet_c((\cR_{u,v})_{\mathbb{F}_q};\overline{\mathbb{Q}}_\ell)$, which by Poincar\'{e} duality corresponds to the Frobenius eigenvalue $q$ in $H^\bullet((\cR_{u,v})_{\mathbb{F}_q};\overline{\mathbb{Q}}_\ell)$.  By standard comparison results, this count is equal to $$\sum_k (-1)^k \dim H^{k,(1,1)}(\cR_{u,v};\mathbb{C}) = (-1)^1 \dim H^{1,(1,1)}(\cR_{u,v};\mathbb{C}) = - \dim H^{1}(\cR_{u,v};\mathbb{C}),$$
where we have used \Cref{prop:van} and \eqref{eq:smooth}.  
\end{proof}

\begin{remark}
    Both \Cref{prop:van} and \Cref{cor:duv} hold for Richardson varieties in arbitrary Kac--Moody groups. 
\end{remark}

\begin{proof}[Proof of \Cref{thm:gj}]
It is known (see, e.g. \cite[Prop.~4.2.1]{riche-soergel-williamson}) that 
\[
\Ext_{\mathcal{O}}^\ast(M_u,M_v) \cong H_c^{\ast+\ell(u,v)}(\cR_{u,v}).
\]
Hence by \Cref{cor:duv}, we have that $\dim \Ext_{\mathcal{O}}^\ast(M_u,M_v)=d_{u,v}$. The combinatorial invariance of this quantity then follows from \Cref{thm:d-equals-g}.
\end{proof}

\section{Connection to cluster structure on Richardson varieties}
\label{sec:cluster}

In this section, we briefly describe the features of the cluster structure on a Richardson variety $\cR_{u,v}$ which motivated the definition of the diamond-generating set $F_{u,v}$ from \Cref{thm:gchain} and derive some consequences for the cluster structure. The proof of the theorem, however, does not depend on this connection. Indeed \Cref{thm:gchain} applies in arbitrary Coxeter groups, while Richardson varieties exist only in crystallographic type, and are known to be cluster varieties only for finite Weyl groups (see \cite{bao2025upperclusterstructurekacmoody} for an \emph{upper} cluster structure in crystallographic type).

Let $W$ be the Weyl group of a complex semisimple Lie group $G$, and let $u,v\in W$. We write $\CC[\cR_{u,v}]$ for the coordinate ring of the Richardson variety $\cR_{u,v}$. A \newword{cluster structure} on $\cR_{u,v}$ includes distinguished elements of $\CC[\cR_{u,v}]$, called \newword{cluster variables}, together with distinguished sets of cluster variables called \newword{clusters}\footnote{In addition to cluster variables and clusters, there is more data required to give a cluster structure, but we shall not discuss it here. See \cite{fomin-zelevinsky-cluster-algebras-1} for a definition of cluster algebra. }. Each cluster contains $\dim \cR_{u,v} =\ell(u,v)$ cluster variables. The cluster variables which are in every cluster are called \newword{frozen variables}. The remaining cluster variables are called \newword{mutable variables}.

\begin{proposition}
\label{prop:cluster-sizes}
    Let $u \leq v$ be elements of a finite Weyl group $W$. The diamond-generating set $F_{u,v}$ is in bijection with the set of frozen variables for the cluster structure on $\mathbb{C}[\cR_{u,v}]$. Consequently, there are $d_{u,v}$ frozen and $\ell(u,v)-d_{u,v}$ mutable variables in each cluster, and the number of these is a combinatorial invariant.
\end{proposition}
\begin{proof}
    The bijection is constructed in the remainder of this section. As noted in the proof of \Cref{thm:gchain}, we have $|F_{u,v}|=d_{u,v}$. This quantity is combinatorially invariant by \Cref{thm:d-equals-g}.
\end{proof}

In order to construct the desired bijection, we now discuss the specifics of the cluster structure introduced in \cite{galashin2025braidvarietyclusterstructures}. Let $\mathbf{v} = s_1\cdots s_{\ell(v)}$ be a reduced expression for $v$. Then \emph{loc.~cit.}~defines a cluster associated to $\mathbf{v}$, which we denote $X_{\mathbf{v}}$. To index the variables in this cluster, they use the notion of a \emph{subexpression} for $u$. We write $\mathbf{v}_{(k)} = s_{1}\cdots s_{k}$ for the $k$th partial product of $\mathbf{v}$. We will similarly describe subexpressions of $\mathbf{v}$ by their partial products; hence a subexpression corresponds to a sequence $\mathbf{u}_{(k)}$ such that $\mathbf{u}_{(0)} = e$, $\mathbf{u}_{(\ell(v))}=u$, and, for all $1\leq k \leq \ell(v)$, $\mathbf{u}_{(k)} = \mathbf{u}_{(k-1)}s_{k}$ or $\mathbf{u}_{(k)} = \mathbf{u}_{(k-1)}$.

\begin{definition}
    The \newword{positive subexpression} of $\mathbf{v}$ is the subexpression $\mathbf{u}^+$ whose partial products are defined recursively by 
    \[ \mathbf{u}^+_{(\ell(v))} = u  \]
    \[ \mathbf{u}^+_{(k-1)} = \min(\mathbf{u}^+_{(k)}, \mathbf{u}^+_{(k)}s_{k}).  \]
    Given a subexpression $\mathbf{u}$, we set $J_{\mathbf{u}} \coloneqq \{k\in [\ell(v)] \mid \mathbf{u}_{(k)} = \mathbf{u}_{(k-1)}\}.$
\end{definition}

The cluster variables in $X_{\mathbf{v}}$ are indexed by $J_{\mathbf{u}^+}$. Among these, the mutable variables correspond to certain ``almost positive subexpressions''. 

\begin{definition}
    For $a\in J_{\mathbf{u}^+}$, let $\mathbf{u}^{\langle a\rangle}_{(\ell(v))} \coloneqq u$ and define  recursively 
    \[\mathbf{u}^{\langle a\rangle}_{(k-1)} \coloneqq \begin{cases}
        \max(\mathbf{u}_{(k)}^{\langle a\rangle}, \mathbf{u}^{\langle a\rangle}_{(k)}s_{k}) &\text{if $k=a$} \\
        \min(\mathbf{u}_{(k)}^{\langle a\rangle},  \mathbf{u}^{\langle a\rangle}_{(k)}s_{k}) &\text{if $k\neq a$} .
    \end{cases}  \]
    We say $a\in J_{\mathbf{u}^+}$ is \newword{mutable} if $\mathbf{u}^{\langle a\rangle}_{(0)} = e$, otherwise $a$ is \newword{frozen}. If $a$ is mutable, then $\mathbf{u}^{\langle a\rangle}$ is called an \newword{almost positive subexpression} of $\mathbf{v}$.
\end{definition}

We now describe how to relate subexpressions of $\mathbf{v}$ to paths in $\widehat\Gamma$.
\newcommand{\pth}{\operatorname{path}}
\begin{definition}
    Let $\mathbf{u}$ be a subexpression of $\mathbf{v}$. Let the indices in  $J_{\mathbf{u}}$ be $i_1< \cdots < i_r$.
    Define $\pth(\mathbf{u})$ to be the path $x_0- x_1 - \cdots - x_r$ in $\widehat\Gamma$ with vertex sequence
    \[ u^{-1} - v^{-1}\mathbf{v}_{(i_r-1)} \mathbf{u}_{(i_r-1)}^{-1}  - v^{-1}\mathbf{v}_{(i_{r-1}-1)}\mathbf{u}_{(i_{r-1}-1)}^{-1} - \cdots - v^{-1}\mathbf{v}_{(i_1-1)}\mathbf{u}_{(i_1-1)}^{-1}  = v^{-1}. \]
\end{definition}

We will be particularly interested in the case where $\pth(\mathbf{u})$ is in fact a path in the directed Bruhat graph $\Gamma$. The following notion is due to Deodhar \cite{deodhar-distinguished}.

\begin{definition}
    A subexpression $\mathbf{u}$ of $\mathbf{v}$ is \newword{distinguished} if, for all $k$ such that $\mathbf{u}_{(k-1)} s_{k} < \mathbf{u}_{(k-1)}$, we have $\mathbf{u}_{(k)} = \mathbf{u}_{(k-1)}s_{k}$.
\end{definition}

The positive and almost positive subexpressions of $\mathbf{v}$ are always distinguished. The positive subexpression is the unique distinguished subexpression $\mathbf{u}$ with $|J_{\mathbf{u}}|=\ell(u,v)$, and the almost positive subexpressions are exactly the subexpressions $\mathbf{u}$ with $|J_{\mathbf{u}}|=\ell(u,v)-2$.

To describe which paths arise as $\pth(\mathbf{u})$ for a distinguished subexpression $\mathbf{u}$, we use reflection orders. A reflection order $\prec$ is said to be \emph{compatible} with the reduced expression $\mathbf{v}$ if $N(v^{-1}\mathbf{v}_{(i)})$ is an initial section of $\prec$ for all $0\leq i \leq \ell(v)$. Equivalently, the smallest $\ell(v)$ positive roots are 
\[ \Phi_1(\mathbf{v}) \prec \Phi_2(\mathbf{v}) \prec \cdots \prec \Phi_{\ell(v)-1}(\mathbf{v}) \prec \Phi_{\ell(v)}(\mathbf{v}), \]
where $\Phi_j(\mathbf{v})$ denotes the root $v^{-1}\mathbf{v}_{(\ell(v)-j+1)}\alpha_{s_{\ell(v)-j+1}}$, 
the unique root $\beta$ with
\begin{equation}
\label{eq:Phi-v}
    \mathbf{v}_{(\ell(v)-j)}^{-1}vt_{\beta} = \mathbf{v}^{-1}_{(\ell(v)-j+1)}v.
\end{equation}


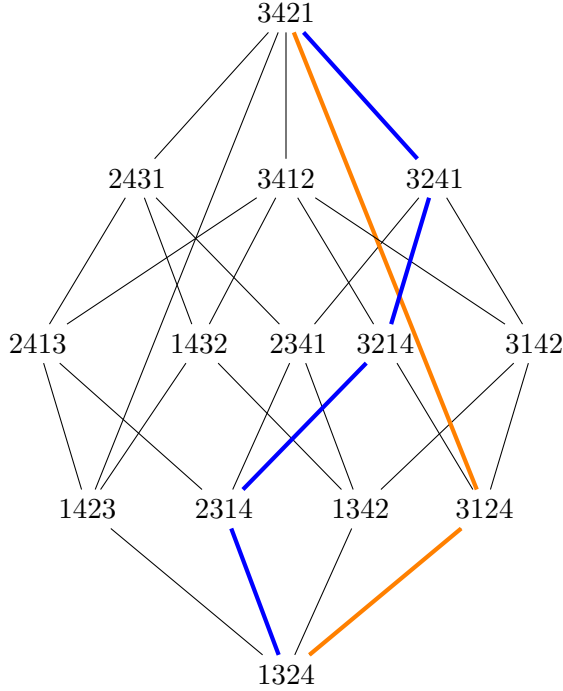
\begin{figure}
    \centering
\begin{tikzpicture}[yscale=.55,xscale=.33]
\node (1324) at (0,4) {1324};
\node (1342) at (3,8) {1342};
\node (1423) at (-8,8) {1423};
\node (2314) at (-2.5,8) {2314};
\node (3124) at (8,8) {3124};
\node (1432) at (-3.5,12) {1432};
\node (2341) at (.5,12) {2341};
\node (2413) at (-10,12) {2413};
\node (3142) at (10,12) {3142};
\node (3214) at (4,12) {3214};
\node (2431) at (-6,16) {2431};
\node (3241) at (6,16) {3241};
\node (3412) at (0,16) {3412};
\node (3421) at (0,20) {3421};
\draw (1324) -- (1342);
\draw (1324) -- (1423);
\draw[blue, ultra thick] (1324) -- (2314);
\draw (1324) -- (3124);
\draw[orange, ultra thick] (1324) -- (3124);
\draw (1342) -- (1432);
\draw (1342) -- (2341);
\draw (1342) -- (3142);
\draw (1423) -- (1432);
\draw (1423) -- (2413);
\draw (1423) -- (3421);
\draw (1432) -- (2431);
\draw (1432) -- (3412);
\draw (2314) -- (2341);
\draw (2314) -- (2413);
\draw[blue, ultra thick] (2314) -- (3214);
\draw (2341) -- (2431);
\draw (2341) -- (3241);
\draw (2413) -- (2431);
\draw (2413) -- (3412);
\draw (2431) -- (3421);
\draw (3124) -- (3142);
\draw (3124) -- (3214);
\draw[orange, ultra thick] (3124) -- (3421);
\draw (3142) -- (3241);
\draw (3142) -- (3412);
\draw[blue, ultra thick] (3214) -- (3241);
\draw (3214) -- (3412);
\draw[blue, ultra thick] (3241) -- (3421);
\draw (3412) -- (3421);
\end{tikzpicture}
    \caption{The Bruhat graph for $[1324, 3421]$ in $S_4$. The longest increasing path $\gamma_0$ is shown in blue and second-length path in orange (using the reduced expressions from \Cref{ex:cluster-example}).}
    \label{fig:cluster-fig}
\end{figure}

Compatible reflection orders exist for every reduced expression, by \Cref{prop:existreflectionorders}. Each reflection order which is Deodhar with respect to $v$ is compatible with a unique reduced expression for $v$. 
\begin{lemma}
\label{lem:path-bijection}
    Let $\prec$ be a reflection order which is compatible with $\mathbf{v}$. Then the map $\mathbf{u}\mapsto \pth(\mathbf{u})$ is a bijection from distinguished subexpressions of $\mathbf{v}$ to the increasing paths in $\Gamma$ from $u^{-1}$ to $v^{-1}$. In particular, $\pth(\mathbf{u}^+)$ is the longest increasing path from $u^{-1}$ to $v^{-1}$.
\end{lemma}
\begin{proof}
    Let $\gamma$ be a path in $\widehat \Gamma$ ending at the vertex $v^{-1}$. Then $\gamma$ may be described as $x_0 \overset{\beta_1}- x_1 \overset{\beta_2}- \cdots \overset{\beta_r}- x_r = v^{-1}$, where
    $x_j = t_{\beta_{j+1}}\cdots t_{\beta_r} x_r$. For this path to have increasing edge labels, we must have $\beta_1\prec \cdots \prec \beta_r$. For each $1\leq i \leq |\Phi^+|$, let $\Phi_i(\prec)$ be the $i$th smallest positive root in the ordering by $\prec$. Let $i_j$ be the index such that $\beta_j = \Phi_{i_j}(\prec)$. Then $\gamma$ has increasing edge labels if and only if $i_1 < i_2 < \cdots < i_r$. Assume now that $\ell(x_{r-1}) < \ell(x_r) $, so that the last edge of $\gamma$ is in fact a directed edge of $\Gamma$. This occurs if and only if $\beta_r \in N(v^{-1})$, which, since $\prec$ is compatible with $\mathbf{v}$, is equivalent to $\{\beta_1,\ldots,\beta_r\}\subseteq N(v^{-1})$. In this case, the element $x_j^{-1} = vt_{\beta_r}t_{\beta_{r-1}}\cdots t_{\beta_{j+1}}$  is obtained by taking the expression $\mathbf{v}$ and deleting the entries in positions $\ell(v)-i_r+1, \ell(v)-i_{r-1}+1,\ldots,\ell(v)-i_{j+1}+1$, by (\ref{eq:Phi-v}). In particular, when $x_0=u^{-1}$, then performing this process for $x_0^{-1}$ gives a subexpression $\mathbf{u}$ of $\mathbf{v}$. This inverts the map $\mathbf{u} \mapsto \pth(\mathbf{u})$.
    Hence the paths $\gamma$ in the undirected Bruhat graph $\widehat \Gamma$ which begin at $u^{-1}$ and end at $v^{-1}$, with increasing edge labels, such that the last edge is in $\Gamma$, are exactly the $\pth(\mathbf{u})$ for subexpressions $\mathbf{u}$ of $\mathbf{v}$.

    We may now verify that $\pth(\mathbf{u})$ is a path in $ \Gamma$ if and only if $\mathbf{u}$ is distinguished. Indeed, assume $\pth(\mathbf{u})$ is not a path in $\Gamma$ and let $j$ an index so that $x_{j-1} \leftarrow x_{j}$. Then $\mathbf{u}_{(\ell(v)-i_j)}s_{\ell(v)-i_j} < \mathbf{u}_{(\ell(v)-i_j)}$ (here we use the fact that $s_{\ell(v)-i_j}s_{\ell(v)-i_j+1}\cdots s_{\ell(v)}$ is a reduced expression). Hence $\mathbf{u}$ is not distinguished, since $\mathbf{u}_{(\ell(v)-i_j+1)}=\mathbf{u}_{(\ell(v)-i_j)}$. Conversely, if $j$ is an index so that $\mathbf{u}_{(\ell(v)-i_j)}s_{\ell(v)-i_j} < \mathbf{u}_{(\ell(v)-i_j)}$ and $\mathbf{u}_{(\ell(v)-i_j+1)}=\mathbf{u}_{(\ell(v)-i_j)}$, then $\ell(x_{j-1})>\ell(x_j).$ This proves the claim.
\end{proof}

\begin{remark}
Under the bijection of \Cref{lem:path-bijection},  Dyer's formula (\Cref{prop:Rpolyincreasingpaths}) for the $R$-polynomial turns into Deodhar's formula in terms of distinguished subexpressions \cite{deodhar-distinguished}.
\end{remark}

\begin{example}
\label{ex:cluster-example}
    Take $W=S_4$, with simple roots $\alpha_1,\alpha_2,\alpha_3$. Let $u=s_{\alpha_2}$ and 
    \[\mathbf{v}=s_{\alpha_2}s_{\alpha_3}s_{\alpha_1}s_{\alpha_2}s_{\alpha_1} = s_1s_2s_3s_4s_5.\] 
    Then $u^{-1}$ and $v^{-1}$ are the permutations with the one-line notations $1324$ and $3421$, respectively. The Bruhat graph $\Gamma_{u^{-1},v^{-1}}$ is shown in \Cref{fig:cluster-fig}. The positive subexpression is
    \begin{center}
        \begin{tabular}{c c c c c c c}
            $\mathbf{v}$ & $=$ & $s_1$ & $s_2$ & $s_3$ & $s_4$ & $s_5$  \\
            $\mathbf{u}^+$ & $ = $ & $1$ & $1$ & $1$ & $s_4$ & $1$ 
        \end{tabular}.
    \end{center}
    Here, the notation indicates that the partial products are given by $\mathbf{u}^+_{(k+1)} = \mathbf{u}^+_{(k)}$ for $k\neq 3$, and $\mathbf{u}^+_{(4)} = \mathbf{u}^+_{(3)}s_4$. In this case $J_{\mathbf{u}^+} = \{1,2,3,5\}$. Among these indices, only $5$ is mutable. The almost positive subexpression $\mathbf{u}^{\langle 5\rangle}$ is 
    
    \begin{center}
        \begin{tabular}{c c c c c c c}
            $\mathbf{u}^{\langle 5\rangle}$ & $ = $ & $s_1$ & $1$ & $s_3$ & $1$ & $s_5$ 
        \end{tabular}.
    \end{center}

    From these two distinguished subexpressions, we produce two paths:
    \begin{align*} \pth(\mathbf{u}^+) &: u^{-1}= s_4 \xrightarrow{\alpha_1} s_5s_4 \xrightarrow{\alpha_2} s_5s_4s_3 \xrightarrow{\alpha_1+\alpha_2+\alpha_3} s_5s_4s_3s_2 \xrightarrow{\alpha_2+\alpha_3}  s_5s_4s_3s_2s_1 = v^{-1} \\
    \pth(\mathbf{u^{\langle 5\rangle}}) &:  u^{-1} = s_5s_3s_1 \xrightarrow{\alpha_1+\alpha_2} s_5s_4s_3s_1 \xrightarrow{\alpha_1+\alpha_2+\alpha_3} s_5s_4s_3s_2s_1 = v^{-1}.
    \end{align*}
    We have labeled the edges by their root labeling. Now, to compare with \Cref{lem:path-bijection}, we should determine a reflection order compatible with $\mathbf{v}$. We have 
    \begin{align*}
        \Phi_1(\mathbf{v}) &= \alpha_1 \\
        \Phi_2(\mathbf{v}) &= \alpha_1+\alpha_2 \\
        \Phi_3(\mathbf{v}) &= \alpha_2 \\
        \Phi_4(\mathbf{v}) &= \alpha_1+\alpha_2+\alpha_3 \\
        \Phi_5(\mathbf{v}) &= \alpha_2+\alpha_3. 
    \end{align*}
    The unique reflection order which puts these roots in order as the smallest roots is
    \[ \alpha_1 \prec \alpha_1+\alpha_2 \prec \alpha_2 \prec \alpha_1 + \alpha_2+\alpha_3 \prec \alpha_2+\alpha_3 \prec \alpha_3. \]
    We can now verify that $\pth(\mathbf{u}^+)$ is the longest increasing path with respect to $\prec$, and $\pth(\mathbf{u}^{\langle 5\rangle})$ is the unique second-length path.
    \renewcommand{\qedsymbol}{$\Diamond$}\qed
\end{example}

We can now state the result motivating the construction in \Cref{thm:gchain}. Note that each element of $J_{\mathbf{u}}$ naturally labels an edge of $\pth(\mathbf{u})$. 

\begin{proposition}
\label{prop:frozen-to-F}
    Let $\prec$ be a reflection order compatible with $\mathbf{v}$, with longest increasing path $\gamma_0 = \pth(\mathbf{u^+})$. Then, under the correspondence between elements of $J_{\mathbf{u}^+}$ and edges of $\pth(\mathbf{u}^+)$, the frozen indices correspond to elements of $F_{u^{-1},v^{-1}}$.
\end{proposition}
\begin{proof}
    An index $a \in J_{\mathbf{u}^+}$ is mutable if and only if $\mathbf{u}^{\langle a \rangle}$ is an almost positive subexpression. By \Cref{lem:path-bijection}, this is true if and only if there is a second-length path diverging from $\gamma_0$ at $v^{-1}\mathbf{v}_{(a)}(\mathbf{u}^+_{(a)})^{-1}$. And by \Cref{lem:countdiverging} this happens if and only if the subsequent edge of $\gamma_0$ is not in $F_{u^{-1},v^{-1}}$.
\end{proof}

\begin{remark}
A clash in standard conventions for increasing paths and for subexpressions has resulted in the inverses appearing in \Cref{prop:frozen-to-F}. However, inversion is a poset isomorphism $[u^{-1},v^{-1}] \to [u,v]$ sending $F_{u^{-1},v^{-1}}$ to a diamond-generating set of $[u,v]$.
\end{remark}


\section*{Acknowledgments}
GTB was supported by National Science Foundation grants DMS-2152991 and DMS-2503536. CG was partially supported by NSF grant DMS-2452032 and by a travel grant from the Simons Foundation. TL was partially supported by NSF grant DMS-2348799 and by a Simons Fellowship.  The authors wish to thank the Sydney Mathematical Research Institute for the excellent collaborative conditions provided during our visit for the ``Modern Perspectives in Representation Theory" program; we also thank Yibo Gao, Chris Hone, Finn Klein, and Geordie Williamson for the stimulating discussions we had there. GTB would also like to thank Amanda Schwartz and David Speyer for helpful conversations.

\bibliographystyle{halpha-abbrv}
\bibliography{guv}
\end{document}